\documentclass[12pt]{article}
\usepackage{e-jc}

\usepackage{pgf,tikz}
\usetikzlibrary{arrows}
\usepackage{amsmath}%
\usepackage{amsfonts}%
\usepackage{amssymb}%
\usepackage{graphicx}

\usepackage{amsthm,amsmath,amssymb}

\usepackage{graphicx}

\usepackage[colorlinks=true,citecolor=black,linkcolor=black,urlcolor=blue]{hyperref}

\theoremstyle{plain}
\newtheorem{theorem}{Theorem}
\newtheorem{lemma}[theorem]{Lemma}

\theoremstyle{definition}

\theoremstyle{remark}

\title{How Ramsey theory can be used to solve Harary's  problem for $K_{2,k}$}
\author{Chula Jayawardene\\
\small  Department of Mathematics\\[-0.8ex]
\small  University of Colombo, Colombo\\[-0.8ex] 
\small Sri Lanka.\\
\small\tt c\_jayawardene@yahoo.com\\
\\
Cecil C. Rousseau and B\'ela Bollob\'as\\
\small  Department of Mathematics\\[-0.8ex]
\small  University of Memphis\\[-0.8ex] 
\small  U.S.A.\\
}
\begin{document}

\maketitle

\begin{abstract}
Harary's conjecture $r(C_3,G)\leq 2q+1$ for every isolated-free graph G with $q$ edges was proved independently by Sidorenko  and Goddard and Klietman.  In this paper instead of $C_3$ we consider $K_{2,k}$ and seek a sharp upper bound for $r(K_{2,k},G)$ over all graphs $G$ with $q$ edges. More specifically if $q\geq 2$, we will show that $r(C_4,G)\leq kq+1$ and that equality holds  if  $G \cong  qK_2$ or $K_3$. Using this we will generalize this result for $r(K_{2,k},G)$ when $k>2$.  We will also show that for every graph $G$ with $q \geq 2$ edges and with no isolated vertices, $r(C_4, G) \leq 2p+ q -  2$ where $p=|V(G)|$ and that equality holds if  $G \cong K_3$.
\end{abstract}

\section*{Introduction}

\vspace{5pt}

At a meeting held at Kent State University in 1980, Harary posed the general problem of determining the relationship between $r(H, G)$ and the sizes (number of edges) of the given graphs. He conjectured that $r(K_3, G) \leq 2q + 1$ for every isolate-free graph $G$ with $q$ edges. This bound is sharp since $r(K_3, T) = 2q + 1$ for any tree $T$ with $q$ edges; also $r(K_3, qK_2) = 2q + 1$. Harary's conjecture was subsequently proved independently by Sidorenko \cite{Sidorenko1991, Sidorenko1993} and by Goddard and Kleitman \cite{Goddard1994s}. More generally, we can take $H$ to be any fixed graph and seek a sharp upper bound for $r(H, G)$ over all graphs $G$ with $q$ edges. In this paper we deal with the case $H=K_{2,k}$ (with $k\geq 2$) and show that these bounds are sharp. We determine all graphs $G$ where the bound is achieved. Prior to this through private communication it is known that the main Theorem had been proved by Bollab{\'a}s and Szemer{\'e}di, using results in extremal graph theory. 

\vspace{15pt}
\section{An upper bound to Ramsey number $r(C_4,G)$}

\begin{theorem}
\label{t1}
For every graph $G$ with $|E(G)|=q \geq 2$ edges and with no isolated vertices, $ $ $   $  $r(C_4, G) \leq 2q + 1$. Equality holds if $G \cong qK_2$ or $K_3$.
\end{theorem}

\begin{proof} As we know that $r(C_4, P_3) = 4$, $r(C_4, K_3) = 7$,  $r(C_4, C_4) = 6$, $r(C_4, 3K_2) = 7$, $r(C_4, K_{1,3}) = 6$,  $r(C_4, K_2 \cup P_3) = 6$, $r(C_4, 2P_3)=7$, $r(C_4, 2K_2) = 5$, $r(C_4, K_{1,4}) = 7$, $r(C_4, K_2 \cup C_3) = 7$, $r(C_4, K_2 \cup K_{1,3}) = 7$,    $r(C_4, 2K_2 \cup P_3) =8$, $r(C_4, K_{1,3}+e) = 7$, $r(C_4, T_3) = 6$ (here $T$ represents the tree on 5 vertices containing exactly one vertex of degree 3), $r(C_4, 4K_2) =9$ (see  \cite{Chlancy1977,Chvatal1972s,Radziszowski2014}). Thus the result holds for $q \in  \{2, 3, 4\}$, with equality for $q \in \{2,4\}$ corresponding to $G \cong qK_2$; with equality for $q =3$ corresponding to $G \cong K_3$ and $G \cong 3K_2$. Let $G$ be a graph with $q \geq 4$ edges and no isolated vertices. If $\Delta(G)=1$  then $G\cong qK_2$. We have $r(C_4, qK_2) \leq 2(q-2)+r(C_4,2K_2) = 2q+1$, and $r(C_4,qK_2) \geq 2q+1$ as a consequence of the two-coloring of $E(K_{2q})$ in which $R \cong  K_{1,2q-1}$. Thus $r(C_4, qK_2) = 2q + 1$ for $q > 1$. Without loss of generality, we can assume $G$ is connected since $r(C_4,G_1 \cup G_2) \leq r(C_4, G_1) + r(C_4, G_2)- 1$. 

\vspace{15pt}

\noindent \textbf{Case 1} If $\Delta(G) = 2$ 

\vspace{7pt}
\noindent Then $G$ a path or a cycle. Using known results, $r(C_4, P_n) \leq r(C_4, C_n) \leq n+2 $ (see \cite{Faudree1974,Goddard1994s,Karolyi2001,Rosta1973}) it follows that $r(C_4, G) \leq 2q$ for all such graphs $G$ with $q \geq 2$ edges and maximum degree $\Delta(G) = 2$. 

\vspace{15pt}

\noindent \textbf{Case 2} If $\Delta(G)  \geq 3$ and $\delta(G) > 1$ 
\vspace{7pt}

\noindent Given a vertex $v$ in $G$ of degree $\Delta(G)$, let $H = G \setminus v$ and let $N_v$, denote the neighborhood of $v$ in $G$.  By induction, then for any isolated vertex free graph $H'$ obtained from $G$ by removing $q'$ edges, we get that $r(C_4, H') \leq 2(q -q') + 1$. 

\vspace{4pt}

\noindent  In the first scenario, suppose that $(R, B)$ is a two-coloring of $E(K_{2q})$ in which there is no red $C_4$ and no blue copy of $G$. We claim that  $\delta(R) \geq 2\Delta(G) - 1$. 
\begin{center}

\begin{tikzpicture}[line cap=round,line join=round,>=triangle 45,x=1.0cm,y=1.0cm]
\clip(-6.799999999999999,0.07999999999999768) rectangle (6.980000000000006,6.359999999999995);
\draw (-3.78,3.76)-- (-1.18,3.14);
\draw (-3.78,3.76)-- (-1.18,1.28);
\draw (-3.78,3.76)-- (-1.16,3.94);
\draw (-4.459999999999998,3.819999999999996) node[anchor=north west] {$w$};
\draw (-3.78,3.76)-- (-1.2,2.26);
\draw (-3.78,3.76)-- (-1.16,4.72);
\draw [shift={(1.8264052525989425,3.403217216852088)},dash pattern=on 5pt off 5pt]  plot[domain=-0.1378721189177181:1.7913798804602392,variable=\t]({1.0*2.4972923401784173*cos(\t r)+-0.0*2.4972923401784173*sin(\t r)},{0.0*2.4972923401784173*cos(\t r)+1.0*2.4972923401784173*sin(\t r)});
\draw [shift={(2.0765988240615103,3.112437810945274)},dash pattern=on 5pt off 5pt]  plot[domain=4.317009984151047:6.259605175855038,variable=\t]({1.0*2.22401944982084*cos(\t r)+-0.0*2.22401944982084*sin(\t r)},{0.0*2.22401944982084*cos(\t r)+1.0*2.22401944982084*sin(\t r)});
\draw [shift={(8.225281854924484,3.3624441608168474)},dash pattern=on 5pt off 5pt]  plot[domain=2.7989393287942046:3.459142694246246,variable=\t]({1.0*7.373955721362437*cos(\t r)+-0.0*7.373955721362437*sin(\t r)},{0.0*7.373955721362437*cos(\t r)+1.0*7.373955721362437*sin(\t r)});
\draw [rotate around={90.0:(2.0,3.089999999999993)},dash pattern=on 5pt off 5pt] (2.0,3.089999999999993) ellipse (1.7016832944913034cm and 0.48869830647425566cm);
\draw (1.0600000000000038,5.579999999999995) node[anchor=north west] {$X=N_G(v)$};
\draw (2.8000000000000043,0.9399999999999973) node[anchor=north west] {Blue copy of graph H };
\draw (-6.579999999999999,5.039999999999996) node[anchor=north west] {$\deg_R(w) \leq2 \Delta(G)-2$};
\draw (3.7800000000000047,5.959999999999995) node[anchor=north west] {$|X| = \Delta(G)-2$};
\draw [->] (5.02,1.2) -- (3.98,1.76);
\draw [rotate around={88.84848110603572:(-1.1600000000000001,3.0899999999999954)}] (-1.1600000000000001,3.0899999999999954) ellipse (2.0844049248583247cm and 0.6189053972730015cm);
\begin{scriptsize}
\draw [fill=black] (-1.2,2.26) circle (1.5pt);
\draw [fill=black] (-1.14,1.28) circle (1.5pt);
\draw [fill=black] (-3.78,3.76) circle (1.5pt);
\draw [fill=black] (-1.18,3.14) circle (1.5pt);
\draw [fill=black] (-1.16,3.94) circle (1.5pt);
\draw [fill=black] (-1.16,4.72) circle (1.5pt);
\draw [fill=black] (2.0,3.6) circle (1.5pt);
\draw [fill=black] (1.98,2.66) circle (1.5pt);
\draw [fill=black] (2.02,4.42) circle (1.5pt);
\end{scriptsize}
\end{tikzpicture}
\end{center}
\begin{center}
\textit{Figure 1: If $w$ is a vertex of with degree $ \leq 2\Delta(G)-2$ in $R$ }
\end{center}

\noindent If $w$ is a vertex of with degree $ \leq 2\Delta(G)-2$ in $R$, we may delete this vertex and its neighborhood in $R$ and still have at least $2q -(2\Delta(G)-1) = 2(q - \Delta(G)) + 1$ vertices.  Thus, there is a blue copy of $H$ in the two colored complete graph that remains after $w$ and its neighborhood in $R$ are deleted. In this copy let $X$ denote the vertex set that plays the role of $N_v$. Since $w$ is adjacent to each vertex of $X$ in $B$ there is a blue copy of $G$, so the claim that $\delta(R) \geq 2\Delta(G) - 1$ is justified. 

\vspace{10pt}
\noindent In the next scenario, since $\Delta(G) \geq 3$, we have $2\Delta(G)-  1 \geq \Delta(G)+ 1$, so $\delta(R) \geq \Delta(G) + 1$. Delete an arbitrary vertex $w$ and exactly $\Delta(G) + 1$ of its neighbors in $R$. Let $Y$ denote the set of $\Delta(G) + 1$ neighbors chosen for deletion. Since $\Delta(G) \geq 3$, the complete graph that remains has at least $2q - (\Delta(G)+2) \geq 2(q - \Delta(G))+1$ vertices, so it must contain a blue copy of $H$. As before, let $X$ be the set that plays the role of $N_v$. Consider the edges between $Y$ and $X$. Since there is no blue copy of $G$, each vertex in $Y$ is adjacent in $R$ to at least one vertex of $X$. Since $|X| = \Delta(G)$ and $|Y| = \Delta(G)+1$ , there must be a vertex $x \in X$ adjacent in $R$ to two or more vertices of $Y$. 

\begin{center}

\begin{tikzpicture}[line cap=round,line join=round,>=triangle 45,x=1.0cm,y=1.0cm]
\clip(-6.799999999999999,0.3399999999999976) rectangle (6.980000000000006,6.2999999999999945);
\draw (-3.86,3.76)-- (-1.18,2.86);
\draw (-3.86,3.76)-- (-1.18,1.28);
\draw (-3.86,3.76)-- (-1.16,3.58);
\draw (-4.459999999999998,3.819999999999996) node[anchor=north west] {$w$};
\draw (-3.86,3.76)-- (-1.18,2.08);
\draw (-3.86,3.76)-- (-1.14,4.7);
\draw [shift={(1.7864052525989422,3.243217216852088)},dash pattern=on 5pt off 5pt]  plot[domain=-0.1378721189177181:1.7913798804602392,variable=\t]({1.0*2.4972923401784173*cos(\t r)+-0.0*2.4972923401784173*sin(\t r)},{0.0*2.4972923401784173*cos(\t r)+1.0*2.4972923401784173*sin(\t r)});
\draw [shift={(2.0824624624624626,3.066366366366366)},dash pattern=on 5pt off 5pt]  plot[domain=4.3064066347353265:6.206932295926013,variable=\t]({1.0*2.1838835809728936*cos(\t r)+-0.0*2.1838835809728936*sin(\t r)},{0.0*2.1838835809728936*cos(\t r)+1.0*2.1838835809728936*sin(\t r)});
\draw [shift={(8.062561174551384,3.34042181309718)},dash pattern=on 5pt off 5pt]  plot[domain=2.8112409549169204:3.4632863976885466,variable=\t]({1.0*7.21255620935653*cos(\t r)+-0.0*7.21255620935653*sin(\t r)},{0.0*7.21255620935653*cos(\t r)+1.0*7.21255620935653*sin(\t r)});
\draw (1.0800000000000038,5.279999999999995) node[anchor=north west] {$X=N_G(v)$};
\draw (2.8600000000000043,0.9199999999999973) node[anchor=north west] {Blue copy of graph H };
\draw (-6.419999999999999,6.119999999999995) node[anchor=north west] {$\deg_R(w) \geq 2\Delta(G) -1$};
\draw (1.760000000000004,4.339999999999995) node[anchor=north west] {$|X| = \Delta(G)$};
\draw [->] (5.02,1.2) -- (3.98,1.76);
\draw [rotate around={89.6350636642702:(-1.15,2.6699999999999986)}] (-1.15,2.6699999999999986) ellipse (1.7483053016917618cm and 0.7691368070268917cm);
\draw (-5.179999999999999,1.759999999999997) node[anchor=north west] {$|Y| = \Delta(G)+1$};
\draw [rotate around={90.0:(1.5000000000000002,3.190000000000004)},dash pattern=on 5pt off 5pt] (1.5000000000000002,3.190000000000004) ellipse (1.4160621461790026cm and 0.4274716386394044cm);
\draw (-4.8199999999999985,5.519999999999995) node[anchor=north west] {$\geq \Delta(G) +1$};
\begin{scriptsize}
\draw [fill=black] (-1.18,2.08) circle (1.5pt);
\draw [fill=black] (-1.18,1.28) circle (1.5pt);
\draw [fill=black] (-3.86,3.76) circle (1.5pt);
\draw [fill=black] (-1.18,2.86) circle (1.5pt);
\draw [fill=black] (-1.16,3.58) circle (1.5pt);
\draw [fill=black] (-1.12,5.3) circle (1.5pt);
\draw [fill=black] (-1.14,4.7) circle (1.5pt);
\draw [fill=black] (1.6,3.44) circle (1.5pt);
\draw [fill=black] (1.56,2.34) circle (1.5pt);
\draw [fill=black] (1.56,4.34) circle (1.5pt);
\end{scriptsize}
\end{tikzpicture}

\end{center}
\begin{center}
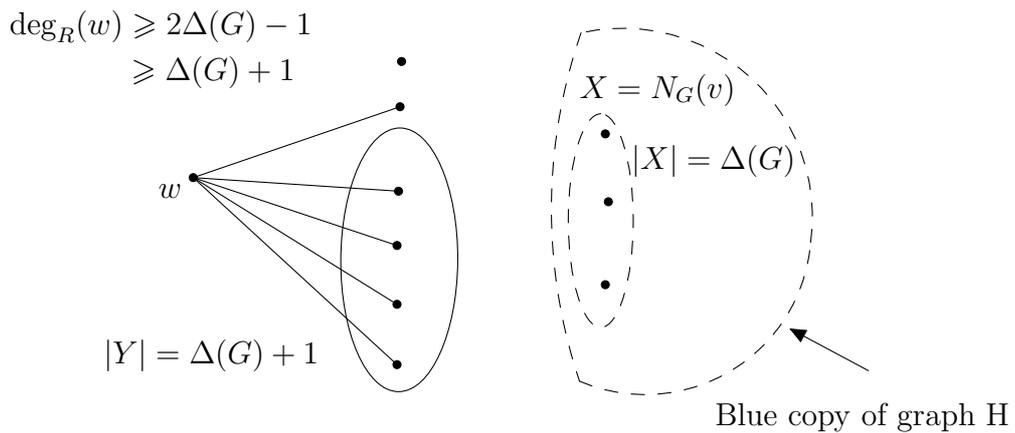
\textit{Figure 2: If $w$ is a vertex of with degree $ \geq 2\Delta(G)-1$ in $R$ }
\end{center}

\noindent Then $x$, $w$ and the two appropriate vertices of $Y$ yield a red $C_4$, a contradiction. 

\vspace{15pt}

\noindent \textbf{Case 3} If $\Delta(G)  \geq 3$ and $\delta(G) = 1$ 
\vspace{7pt}

\noindent It should be also noted that in the case $\Delta(G)  \geq 3$ and $\delta(G) = 1$
works though unless $G \setminus v$ has  isolated vertices.  However, connectivity of $G$ together with
the inequality $r(C_4, T_{q+1}) \leq \max \{4; q + 2; r(C_4,K_{1,q}))\leq 2q + 1$ (see \cite{Baskoro2006s,Erdos1988s}) gives that $G$ is a connected graph  which is not a tree.

\vspace{12pt}
\noindent Therefore, we are left with the case when $G$ has a vertex $v$ of degree $\Delta(G)$ such that $v$ is adjacent to $s$ pendent vertices of $G$, where $1 \leq s \leq q-3$. Let $H$ be the graph obtained by removing the $s$ pendent vertices from $G$. 

\vspace{12pt}
\noindent Suppose that $(R, B)$ is a two-coloring of $E(K_{2q})$ in which there is no red $C_4$ and no blue copy of $G$.  We claim that that $\Delta(R) \leq s$. If $w$ is a vertex of degree $\geq s+1$, let $Y$ consist of any $s+1$ red neighbors of $w$. Define $X=Y \cup \{w\}$. By induction hypothesis, there is a blue copy of $H$ in the two colored complete graph that remains after the vertices of $Y$ are deleted. In this copy let $x$ denote the vertex
set that plays the role of $N_v$. As before then $H$ can be extended to a blue $G$ as in order to avoid a red $C_4$,  $s$ vertices of $Y$ will be forced to be adjacent to $x$ in blue. Therefore, $\Delta(R) \leq s$. Let $w$ be a vertex with degree $\Delta(R)\leq s$. Let $Y$ be a set containing $N_R(w)$ along with $2s-1-\Delta(R)$ other vertices distinct from $w$. Define $X=Y \cup \{w \}$. 

\begin{center}

\begin{tikzpicture}[line cap=round,line join=round,>=triangle 45,x=1.0cm,y=1.0cm]
\clip(-8.5,1.1399999999999972) rectangle (6.160000000000005,7.619999999999994);
\draw (-3.5,3.56)-- (-1.22,2.28);
\draw (-4.219999999999999,4.499999999999995) node[anchor=north west] {$w$};
\draw (-3.5,3.56)-- (-1.18,3.22);
\draw (-3.5,3.56)-- (-1.16,4.28);
\draw (-6.18,2.899999999999995) node[anchor=north west] {$\deg_R(w) =\Delta(R)\leq s$};
\draw [dash pattern=on 5pt off 5pt] (-3.5,3.56)-- (-1.12,5.3);
\draw [dash pattern=on 5pt off 5pt] (-3.5,3.56)-- (-1.12,5.76);
\draw [dash pattern=on 5pt off 5pt] (-3.5,3.56)-- (-1.1,6.5);
\draw [rotate around={89.58179005850387:(-1.12,4.279999999999997)}] (-1.12,4.279999999999997) ellipse (2.807530703354169cm and 0.6117423070838824cm);
\draw (-3.5,3.56)-- (-1.14,3.78);
\draw (-8.200000000000001,6.379999999999994) node[anchor=north west] {$\deg_B(w) =2s-1-\Delta(R)\geq s-1$};
\draw (0.22000000000000236,3.5999999999999948) node[anchor=north west] {In the blue copy of };
\draw (3.780000000000004,3.5799999999999947) node[anchor=north west] {$H, x$};
\draw (0.28000000000000236,2.999999999999995) node[anchor=north west] {play the role of };
\draw (3.1400000000000032,2.999999999999995) node[anchor=north west] {$v \in G.$};
\draw (2.480000000000003,5.139999999999994) node[anchor=north west] {$x$};
\draw [shift={(-0.06332048104287226,-1.3017323051500944)},dash pattern=on 5pt off 5pt]  plot[domain=1.1699174252529474:2.1861273262842222,variable=\t]({1.0*5.953755715761224*cos(\t r)+-0.0*5.953755715761224*sin(\t r)},{0.0*5.953755715761224*cos(\t r)+1.0*5.953755715761224*sin(\t r)});
\begin{scriptsize}
\draw [fill=black] (-1.18,3.22) circle (1.5pt);
\draw [fill=black] (-1.18,2.22) circle (1.5pt);
\draw [fill=black] (-3.5,3.56) circle (1.5pt);
\draw [fill=black] (-1.12,5.3) circle (1.5pt);
\draw [fill=black] (-1.16,4.28) circle (1.5pt);
\draw [fill=black] (-1.1,6.5) circle (1.5pt);
\draw [fill=black] (-1.12,5.76) circle (1.5pt);
\draw [fill=black] (-1.14,3.78) circle (1.5pt);
\draw [fill=black] (2.28,4.16) circle (1.5pt);
\end{scriptsize}
\end{tikzpicture}

\end{center}
\begin{center}
\textit{Figure 3: If $w$ is a vertex of red degree at least two }
\end{center}

\vspace{7pt}
\noindent By induction hypothesis, there is a blue copy of $H$ in the two colored complete graph that remains after the vertices of $Y$ are deleted. If $s>1$ by the above argument, $x$ can be adjacent in blue to at most $s-2$ vertices of $Y$  and can be adjacent in red to at most $s$ vertices in $Y$(as $ \Delta(R)\leq s$). A contradiction as $|Y|=2s-1$. If $s=1$ as $(w,x)$ is blue, we get a copy of $G$, a contradiction.

\vspace{10pt}
\noindent Thus $r(C_4, G) \leq 2q$ for every graph $G$ with $q$ edges and with no isolated vertices other than $2K_2$, and the proof is complete. 
\end{proof}
\vspace{15pt}

\section{An upper bound to Ramsey number $r(C_4,G)$ if $G$ is connected.}

\vspace{10pt}

\begin{theorem}
\label{t2}
For every isloated vertex free graph $G$ with $q \geq 2$ edges and $p \geq 3$ vertices,  $ $ $r(C_4, G) \leq 2p+ q -  2$. Equality holds if $G \cong K_3$.
\end{theorem}

\begin{proof} It is easy to verify the theorem for $q \leq 4$ by the results of the previous 
section. If $\Delta(G) \leq 2$ then $G$ is a path or a cycle, and using known results, again 
it is easy to verify that the theorem is true in this case. So it suffices to show 
$r(C_4, G) \leq q + 2p - 3$, if $q \geq 5$. Now assume $\Delta(G) \geq 3$. Given a vertex $v$ in $G$ of degree $\Delta(G)$, let $H =G \setminus v$ and let $N_v$, denote the neighborhood of $v$ in $G$. First assume that $H$ has no isolated vertices. By induction, then for any isolated vertex free graph $H'$ obtained from $G$ by removing $q'$ edges (and the $p'$ corresponding vertices), we get that  $r(C_4, H') \leq (q -q') + 2(p-p')-3$. In particular, $r(C_4, H) \leq (q-\Delta(G) ) +2(p-1)-3$. Suppose that $(R, B)$ is a two-coloring of $E(K_{q+2p-3})$ in which there is no red $C_4$ and no blue copy of $G$. then using a similar argument as in the last proof we would get a contradition for all possible cases expect the third case when $\Delta(G)  \geq 3$ and $\delta(G) = 1$.

\vspace{7pt}

\noindent In this case let $v\in G$ represent vertex of degree 1 in $G$, and $N_v$ consist of $v_1$. Clearly,   $|\delta(R)| \geq 1$. First assume that, $|\Delta(R)|=1$. Let $u \in K_{2q}$  be a vertex with degree 1 and suppose it it adjacent to $w$ in red.  Then the graph  obtained by removing $u$ and $w$ from $R$ (say $K \cong K_{2(q-1)}$)  by induction hypothesis will have a blue copy of $G \setminus v$ in it. As before let $X=\{x\}$ be the set that play the role of $N_v$. As $u$ is adjacent to all vertices of $K$ we will get a blue copy of  $G$ in $K_{2q}$. Therefore, we may assume that $|\Delta(R)|>1$. Let $u \in K_{2q}$  be a vertex with at least two neighbors in red say $w_1$ and $w_2$. Let $Y =\{ w_1, w_2 \}$. Then as before the complete graph obtained by removing $Y\cup \{u\}$ will contain a blue copy of $G\setminus v$ as illustrated in the following figure(since $H$ has $p-1$ vertices and $q-1$ edges).

\noindent Since $|N_v| = 1$, $|Y| = 2$ and there is no blue copy of $G$, the vertex $x$ must be adjacent in $R$ to the vertices $w_1$ and $w_2$. But then $u w_1 x w_2 x$ will yield a red $C_4$, a contradiction.

\begin{center}

\begin{tikzpicture}[line cap=round,line join=round,>=triangle 45,x=1.0cm,y=1.0cm]
\clip(-6.799999999999999,0.15999999999999764) rectangle (7.000000000000006,6.2999999999999945);
\draw (-3.72,3.46)-- (-1.22,2.28);
\draw (-4.459999999999998,3.819999999999996) node[anchor=north west] {$w$};
\draw (-3.72,3.46)-- (-1.18,3.22);
\draw (-3.72,3.46)-- (-1.14,4.7);
\draw [shift={(1.7864052525989422,3.243217216852088)},dash pattern=on 5pt off 5pt]  plot[domain=-0.1378721189177181:1.7913798804602392,variable=\t]({1.0*2.4972923401784173*cos(\t r)+-0.0*2.4972923401784173*sin(\t r)},{0.0*2.4972923401784173*cos(\t r)+1.0*2.4972923401784173*sin(\t r)});
\draw [shift={(2.0824624624624626,3.066366366366366)},dash pattern=on 5pt off 5pt]  plot[domain=4.3064066347353265:6.206932295926013,variable=\t]({1.0*2.1838835809728936*cos(\t r)+-0.0*2.1838835809728936*sin(\t r)},{0.0*2.1838835809728936*cos(\t r)+1.0*2.1838835809728936*sin(\t r)});
\draw [shift={(8.062561174551384,3.34042181309718)},dash pattern=on 5pt off 5pt]  plot[domain=2.8112409549169204:3.4632863976885466,variable=\t]({1.0*7.21255620935653*cos(\t r)+-0.0*7.21255620935653*sin(\t r)},{0.0*7.21255620935653*cos(\t r)+1.0*7.21255620935653*sin(\t r)});
\draw (1.0800000000000038,5.279999999999995) node[anchor=north west] {$X=\{x\}$};
\draw (2.3200000000000043,0.7199999999999974) node[anchor=north west] {Blue copy of graph G \textbackslash  v};
\draw (-6.179999999999999,4.8799999999999955) node[anchor=north west] {$\deg_R(w) \geq 2$};
\draw [->] (5.02,1.2) -- (3.98,1.76);
\draw [rotate around={89.6350636642702:(-1.15,2.6699999999999986)}] (-1.15,2.6699999999999986) ellipse (1.7483053016917618cm and 0.7691368070268917cm);
\draw (-4.679999999999998,2.0999999999999965) node[anchor=north west] {$|Y| =\{w_1,w_2\}$};
\draw [rotate around={90.0:(1.58,3.2099999999999977)},dash pattern=on 5pt off 5pt] (1.58,3.2099999999999977) ellipse (1.4160621461789815cm and 0.4274716386393988cm);
\draw (-1.18,3.22)-- (1.58,3.2);
\draw (1.58,3.2)-- (-1.18,2.22);
\draw (-1.2599999999999971,4.119999999999996) node[anchor=north west] {$w_1$};
\draw (-1.3199999999999972,2.2199999999999966) node[anchor=north west] {$w_2$};
\draw (1.400000000000004,4.019999999999996) node[anchor=north west] {$x$};
\begin{scriptsize}
\draw [fill=black] (-1.18,3.22) circle (1.5pt);
\draw [fill=black] (-1.18,2.22) circle (1.5pt);
\draw [fill=black] (-3.72,3.46) circle (1.5pt);
\draw [fill=black] (-1.12,5.3) circle (1.5pt);
\draw [fill=black] (-1.14,4.7) circle (1.5pt);
\draw [fill=black] (1.58,3.2) circle (1.5pt);
\end{scriptsize}
\end{tikzpicture}

\end{center}
\begin{center}
\textit{Figure 4: If $w$ is a vertex of red degree at least two }
\end{center}

\noindent Thus, $r(C_4, G) \leq q + 2p - 3$ for every graph $G$ with $q$ edges and with no isolated vertices other than $K_3$, and the proof is complete. 
\end{proof}

\vspace{8pt}

\section{An Upper Bound for the Ramsey Number $r(K_{2,k}, G)$ where $k\geq 2$}

\begin{lemma} $r(K_{2,k}, G) \leq kq + 1$ if $G$ is a path, star or triangle.
\end{lemma}

\noindent 

\begin{proof}\noindent We showed before that  $r(K_{2,2}, K_{1,q})=r(C_4, K_{1,q}) \leq 2q + 1$ if $q \geq 2$. Also \cite{Haggkvist1989} gives
 $r(K_{2,k}, K_{1,2}) = r(K_{2,k}, P_3) \leq 2k + 1$. And thus we can conclude the result is true 
for stars as  $r(K_{2,k}, K_{1,q}) \leq r(B_k,K_{1,q}) \leq \max\{2q + 1, \lfloor \frac{3}{2}(k + q-1) \rfloor +1 \}\leq kq+1$ for $k \geq 3$ and $q \geq 3$. For paths and triangles the result follows directly from  $r(K_{2,k}, P_{q+1}) \leq k+q+1$ \cite{Haggkvist1989} and $r(B_k, B_1)\leq  2k + 3$ \cite{Rousseau1978s} respectively. 
\end{proof}

\begin{lemma}  For every isolated vertex free graph $G$ with $q \geq 2$ edges, $r(K_{2,k}, G) \leq kq +1$ if $k \geq 2$.
\end{lemma}

\begin{proof} We will use induction on $q$. The result is true for $k = 2$ or $q \in \{2, 3\}$. Also without loss of generality $G$ is connected and the result is true if $G$ is a path, star or triangle (these follow from previous lemmas). Thus we can restrict our attention to $G$ such that $G$ is not a path, star or triangle satisfying $\Delta(G)>1$.  Using the argument in Theorem 1:case 3, we may assume that $\delta >1$(by considering the two cases, $\Delta(R)\leq s+k-2$ and $\Delta(R)> s+k-2$). Thus, given a vertex in $G$ of degree $\Delta(G)$, let $H = G \setminus v$ and let $N_v$ denote the neighborhood of $v$ in $G$. Then $H$ has $q-\Delta$ edges. Suppose that $(R, B)$ is a two-coloring of $E(K_{kq+1}$) in which there is no red $K_{2,k}$ and no blue copy of $G$. We claim that that $\delta(R) \geq (k- 1)\Delta + 1$. If $w$ is a vertex of with degree $\leq(k -1)\Delta$ in $R$, we may delete this vertex and its neighborhood in $R$ and still have at least $kq+1-((k- 1)\Delta +1) \geq \max \{ k(q- \Delta)+1,p- 1 \}$ vertices(as $G$ is not a star). Thus there is a blue copy of $H$ in the two-colored complete graph that remains after $w$ and its neighborhood in $R$ are deleted. In this copy let $X$ denote the vertex set that plays the role of $N_v$. Since $w$ is adjacent to each vertex of $X$ in $B$ there is a blue copy of $G$, so the claim that $\delta(R) \geq (k -1)\Delta +1$ is justified. Delete an arbitrary vertex $w$ and exactly $(k -1)\Delta + 1$ of its neighbors in $R$. Let $Y$ denote the set of $(k -1)\Delta + 1$  neighbors chosen for deletion. As before the complete graph that remains has at least $\max \{ k(q-\Delta) + 1,p -1\}$ vertices, so it must contain
 a blue copy of $H$. As before, let $X$ be the set that plays the role of $N_v$. Consider
 the edges between $Y$ and $X$. Since there is no blue copy of $G$, each vertex in $Y$ is
 adjacent in $R$ to at least one vertex of $X$. Since $|X|=\Delta$ and there must be a  vertex $x \in X$ adjacent in $R$ to $k$ or more vertices of $Y$. Then $x$  and the $k$ appropriate vertices of $Y$ yield a red $K_{2,k}$, a contradiction. 
\end{proof}

\begin{theorem} For every graph $G$ with no isolated vertices, $r(K_{2,k},G) \leq kq + 2$ if $k \geq 3$ and equality holds if  $G \cong K_2$.
\end{theorem}

\begin{proof}
\noindent This result follows from the previous lemma together $r(K_{2,k}, K_2) = k+2$. 
\end{proof}

\end{document}